%% file: main.tex
\documentclass[12pt,a4paper]{article}

\include{definitions}
\title{An efficient method for the anisotropic diffusion equation in magnetic fields.}

\author{Dean Muir, Kenneth Duru, Matthew Hole, Stuart Hudson}

\date{}

\begin{document}

\maketitle

\begin{abstract}
    We solve the anisotropic diffusion equation in 2D, where the dominant direction of diffusion is defined by a vector field which does not conform to a Cartesian grid. Our method uses operator splitting to separate the diffusion perpendicular and parallel to the vector field. The slow time scale is solved using a provably stable finite difference formulation in the perpendicular to the vector field, and an integral operator for the diffusion parallel to it. Energy estimates are shown to for the continuous and semi-discrete cases. Numerical experiments are performed showing convergence of the method, 
    and examples is given to demonstrate the capabilities of the method.
\end{abstract}


\section{Introduction}\label{sec:Introduction}

The anisotropic diffusion equation provides a simplified model for transport phenomena in magnetic confinement fusion devices. These devices use extremely strong magnetic fields to confine a super heated plasma using, typically a few million times stronger than the Earth's magnetic field. The strength of the magnetic field results in diffusive processes being orders of magnitude faster along magnetic field lines compared to across them. The ratio of diffusion coefficients parallel and perpendicular to the field lines can exceed $\sim 10^{10}$. This disparity results in numerical errors quickly polluting the solution when the computational grid is not aligned with the magnetic field line \cite{gunter_modelling_2005}. 

\citeauthor{gunter_modelling_2005}~\cite{gunter_modelling_2005} resolved this issue by introducing a method dependent on tracing the magnetic field lines, resulting in a field aligned form of the anisotropic diffusion equation that minimises the numerical pollution~\cite{gunter_modelling_2005}.
\citeauthor{hudson_temperature_2008}~\cite{hudson_temperature_2008} showed that the isocontours of steady state solutions of the equation closely resemble features of the underlying field. This suggests that solutions to the equation can provide a proxy to other properties important for the confinement of particles \cite{hudson_temperature_2008,helander_heat_2022,paul_heat_2022}.
However, an equilibrium solver can be undesirable even when these steady state solutions are sought, since as the perpendicular diffusion vanishes, the problem becomes ill conditioned at best and ill-posed at worst. To resolve this, \citeauthor{chacon_asymptotic-preserving_2014}~\cite{chacon_asymptotic-preserving_2014} introduced a time dependent method using on operator splitting and replacing the parallel diffusion term with an integral operator formulated in earlier work by \citeauthor{del-castillo-negrete_local_2011}~\cite{del-castillo-negrete_local_2011} for the parallel diffusion.

In this paper we introduce an approach to solving a field aligned form of the anisotropic diffusion equation which is provably stable and efficient. We demonstrate this on a simplified 2D version of the problem where we consider one spatial dimension lying purely parallel to the magnetic field and the other being perpendicular to it. For simplicity, this paper replaces the magnetic field with functions for the parallel map. This simplification captures many of the challenges associated with the full 3D problem and works when the field line tracing is used. We derive energy estimates of the solution of the underlying initial boundary value problem (IBVP).  In the perpendicular direction we approximate the diffusion equation using summation-by-parts (SBP) finite difference operators \cite{nordstrom_summation-by-parts_2016}. Boundary conditions and the parallel diffusion term are implemented weakly using the simultaneous approximation term (SAT). We prove numerical stability by deriving discrete energy estimates mimicking the continuous energy estimates. The numerical method can be extended to multiple dimensions and complex geometries.

This paper will be ordered as follows.
In \S\ref{sec:preliminary} we detail the summation by parts formulation, which is used to discretise perpendicular to the magnetic field. Section \S\ref{sec:the anisotropic diffusion equation} introduces the field aligned anisotropic diffusion equation formally, details the simplifications made in this paper to reduce it to 1 dimension by introducing an integral operator for the parallel transport and provides a proof of well-posedness.
In \S\ref{sec:Numerical Approach} we introduce the semi-discrete form of the anisotropic diffusion equation using the summation by parts with simultaneous approximation terms (SBP-SAT), and the discrete form of the parallel integral operator. We also prove stability for the semi-discrete problem. The numerical approach for the discrete problem is outlined in S\ref{sec:sub:The fully-discrete approximation}.
Numerical results are presented in \S\ref{sec:Results}. This includes demonstrating convergence by the method of manufactured solutions, followed by some examples which illustrates the effects of the parallel map and the robustness of the method. We summarise the paper in section \S\ref{sec:Conclusions}.

\section{Preliminaries}\label{sec:preliminary}
Here we introduce the summation by parts formulation, which gives the provably stable finite difference scheme used in this work.
We consider the spatial interval $x \in [0, L]$ and discretise it into $n$ grid points with a uniform spatial step $\Delta{x} >0$, having
$$
x_j =  (j-1)\Delta{x}, \quad \Delta{x} = \frac{L}{n-1}, \quad j = 1, 2, \cdots n,
$$
and  $\mathbf{u} = [u_1(t), u_2(t), \cdots u_n(t)]^T\in\R^n$ denotes the semi-discrete scalar field on the grid. Let $D_x, D_{xx}^{(k)} \in \mathbb{R}^{n\times n}$ denote discrete approximations of the first and second spatial derivatives on the grid,  that is $(D_x\mathbf{u})_j \approx \partial u/\partial x|_{x=x_j}$ and $(D_{xx}^{(k)}\mathbf{u})_j \approx \partial\left(\kappa\partial u/\partial x\right)/\partial x|_{x=x_j}$, where $\kappa >0$ is the diffusion coefficient.
The discrete operators $D_x, D_{xx}^{(k)}$ are called SBP operators if
\begin{align}\label{eq:sbp_x}
    & D_x=H^{-1}Q,  \quad  Q+Q^T= B:=\emph{diag}([-1,0,\cdots,1]), \\
    &H=H^T, \quad \mathbf{u}^T H \mathbf{u} > 0, \quad \forall \mathbf{u}\in\R^n,\\
\label{eq:sbp_xx}
    & D_{xx}^{(k)}  = H^{-1}(-M^{(k)} + BKD_x),  \quad M^{(k)} = (M^{(k)})^T, \quad\mathbf{u}^TM^{(k)} \mathbf{u}\geq 0,
\end{align}
where $K=\emph{diag}([\kappa(x_1),\kappa(x_2),\cdots,\kappa(x_n)])$.
The SBP operators $D_x$ and  $D_{xx}^{(k)}$ are called \textit{fully compatible} if
    \begin{align}\label{defn:eq:fully compatible SBP operator}
        M^{(k)} = D_x^T \left(KH\right) D_x + R_x^{(k)},  \quad R^{(k)} = (R^{(k)})^T, \quad\mathbf{u}^TR^{(k)} \mathbf{u}\geq 0.
    \end{align}
We will use fully compatible and diagonal norm SBP operators with 
$H=\Delta{x}\emph{diag}([h_1,h_2,\cdots,h_n])$, where $h_j >0$ are the weights of a composite quadrature rule. The SBP properties \eqref{eq:sbp_x}--\eqref{eq:sbp_xx} will be useful in proving numerical stability.

\section{The anisotropic diffusion equation}\label{sec:the anisotropic diffusion equation}

The field aligned anisotropic diffusion equation \cite{gunter_modelling_2005,hudson_temperature_2008} is given by,
\begin{align}\label{eq:ADE Field Aligned}
    \pdv{u}{t} &= \nabla\cdot(\kappa_\perp\nabla_\perp u) + \nabla\cdot(\kappa_\parallel\nabla_\parallel u),
\end{align}
where $\nabla_\parallel$ is the directional derivative along the magnetic field, $\nabla_\perp=\nabla-\nabla_\parallel$, $\kappa_\perp >0$ and $\kappa_\parallel >0$ are the diffusion coefficients in the perpendicular and parallel directions, respectively. Note that $\kappa_\parallel/\kappa_\perp \gg 1$ and can exceed $\sim 10^{10}$ in many relevant applications. Equation \eqref{eq:ADE Field Aligned} is fully 3D in space. To simplify we follow previous works outlined in \S\ref{sec:Introduction}, and solve equation \eqref{eq:ADE Field Aligned} on a 2D plane in the perpendicular ($\bm{e}_1$ and $\bm{e}_2$) direction, which reduces the computational complexity significantly. The effect of the parallel diffusion is then included through an integral operator $\mathcal{P}_\parallel$. This gives,
\begin{align}\label{eq:ADE Field Aligned 2D}
    \pdv{u}{t} &= \nabla\cdot(\kappa_\perp\nabla_\perp u) + \mathcal{P}_\parallel u,
\end{align}
where 
\begin{align}\label{eq:parallel operator continuous}
    \nabla\cdot(\kappa_\parallel\nabla_\parallel) \sim \mathcal{P}_\parallel, \quad 
    \left(u, \left(\mathcal{P}_\parallel + \mathcal{P}_\parallel^{\dagger}\right) u \right)\leq 0.
\end{align}
Here $\mathcal{P}_\parallel^\dagger$ is the adjoint operator and $\left(\cdot, \cdot\right)$ denotes the standard $L_2$ scalar product defined on the 2D plane.
The operator $\mathcal{P}_\parallel$ can be constructed explicitly, for instance with a Green's function \cite{del-castillo-negrete_local_2011,del-castillo-negrete_parallel_2012,chacon_asymptotic-preserving_2014}, however this work will keep much more closely with references which use interpolation \cite{gunter_modelling_2005,hudson_temperature_2008}. 
In particular, we suppose $\mathcal{P}_\parallel\propto\mathcal{P}_f+\mathcal{P}_b$, where $\mathcal{P}_f$ and $\mathcal{P}_b$ are operators which trace the solution $u$ onto the "forward" (positive along magnetic field) and "backward" (negative along magnetic field) planes and also mimic the diffusive integral operators so that $\mathcal{P}_fu=w_f$, $\mathcal{P}_bu=w_b$ and $\norm{\mathcal{P}_f},\norm{\mathcal{P}_b}\leq1$. 
The purely parallel solution is then an average of the two projected values, so that $u_\parallel = 1/2(w_f+w_b)$.

We will also make the further simplification in this paper that $u$ is constant in $\bm{e}_1$ to reduce the number of dimensions to 2D, with solutions now in 1D. Thus \eqref{eq:ADE Field Aligned} reduces to,
\begin{align}\label{eq:ADE 1.5D}
  \pdv{u}{t}    &= \pdv{x}\left( \kappa \pdv{u}{x} \right)  + \mathcal{P}_{\parallel}u, \quad x \in [0, L], \quad \kappa = \kappa_\perp >0,
\end{align}
with smooth initial condition
\begin{align}\label{eq:initial condition} 
    u(x,0) = f(x).
\end{align}
For simplicity we will also only consider the case of Neumann boundary conditions,
\begin{align}\label{eq:neumann boundary conditions}
    \kappa \eval{\pdv{u}{x}}_{x=0} = g(t), \quad \kappa \eval{\pdv{u}{x}}_{x = L} = g(t).
\end{align}
Specifically, we will use no-flux boundaries, so that $g(t)=0$ in the analysis and examples going forward.
The following theorem proves the well-posedness of the simplified problem.
\begin{theorem}\label{eq:stability ibvp}
Consider the anisotrpic diffusion equation \eqref{eq:ADE 1.5D} subject to the smooth initial condition
\eqref{eq:initial condition} and boundary conditions \eqref{eq:neumann boundary conditions}. If $\left(u, \left(\mathcal{P}_\parallel + \mathcal{P}_\parallel^{\dagger}\right) u \right)\leq 0$ then
$$
\dv{t}\norm{{u}}^2 \le 0.
$$
\end{theorem}
\begin{proof}
   We use the energy method, that is we multiply \eqref{eq:ADE 1.5D} with the solution $u$ and integrate over the domain
\begin{align}
    \int_{0}^{L} u\pdv{u}{t} \dd x   &= \int_{0}^{L} u\pdv{x}\left( \kappa \pdv{u}{x} \right) \dd x + \int_{0}^{L} u \mathcal{P}_\parallel u \dd x.    
\end{align} 
Integration by parts gives
\begin{align}
    \frac{1}{2}\dv{t}\int_{0}^{L} u^2 \dd x     &= -\int_{0}^{L} \pdv{u}{x}\kappa\pdv{u}{x} \dd x + \left[ u\kappa\pdv{u}{x} \right|_{0}^{L} + \int_{0}^{L} u \mathcal{P}_\parallel u \dd x.    \label{eq:Energy est continuous IBP}
\end{align} 
Enforcing the boundary conditions \eqref{eq:neumann boundary conditions} and adding the conjugate transpose of the product gives
\begin{align}
    \dv{t}\norm{{u}}^2    &= -2\int_{0}^{L} \pdv{u}{x}\kappa\pdv{u}{x} \dd x   + \int_{0}^{L} u \left(\mathcal{P}_\parallel + \mathcal{P}_\parallel^{\dagger}\right) u \dd x \leq 0.    \label{eq:Energy est continuous IBP_final}
\end{align} 
\end{proof}
To ensure stability of the numerical method we will seek to mimic the energy estimate \eqref{eq:Energy est continuous IBP_final} at the discrete level.
\section{Numerical approach}\label{sec:Numerical Approach}
We will follow the method of lines by discretising the spatial variable while leaving the time variable continuous. We will approximate the spatial derivative using SBP operators \cite{mattsson_summation_2004}, while the boundary conditions and the parallel operator will be implemented weakly using penalties.
The semi-discrete approximation of the  anisotropic diffusion equation \eqref{eq:ADE 1.5D} using the SBP-SAT method is
\begin{align}\label{eq:semi-discrete approximation}
     \dv{\bm{u}}{t}
        &=  D_{xx}^{(k)}\bm{u} +  \,\text{SAT} +  P_\parallel \bm{u}, \quad \bm{u}(0) = \bm{f},
\end{align}
where $D_{xx}^{(k)}$ is the SBP operator given in \eqref{eq:sbp_xx} and
\begin{align}\label{eq:discrete parallel operator and SAT}
    P_\parallel = \frac{\tau_\parallel}{2}H^{-1}\kappa_\parallel  \overbrace{\left(I - \frac{1}{2}[P_f + P_b]\right)}^{A_\parallel}, \quad  \text{SAT} = \tau_0 H^{-1} B\left(KD_x \bm{u} - \bm{g}\right),
\end{align}
are weak numerical implementations of the parallel diffusion operator \eqref{eq:parallel operator continuous} and the boundary conditions \eqref{eq:neumann boundary conditions}, $\tau_\parallel$ and $\tau_0$ are penalty parameters to be determined by requiring stability. 
Before showing stability we first prove the following lemma regarding the definiteness of the numerical parallel diffusion operator.
\begin{lemma}\label{lem:numerical parallel operator}
Consider the numerical parallel diffusion operator
\begin{align}\label{eq:discrete parallel operator}
    P_\parallel = \frac{\tau_\parallel}{2}H^{-1}\kappa_\parallel  A_\parallel, \quad A_\parallel = I - \frac{1}{2}[P_f + P_b],
\end{align}
with $\kappa_\parallel\ge 0$, $\tau_\parallel = {\alpha}/{\Delta x}$ and $\alpha \le 0$. If $\|P_f\|\le 1$ and $\|P_b\|\le 1$ then
$$
\bm{u}^T \left(A_\parallel + A_\parallel^T\right) \bm{u} \ge 0
, \quad
\bm{u}^T \left(\left(HP_\parallel\right) + \left(HP_\parallel\right)^T\right) \bm{u} 
\le 0, \quad \forall \bm{u}\in \mathbb{R}^n.
$$
\end{lemma}
\begin{proof}
The sum $A_\parallel + A_\parallel^T$ is symmetric.
Since $\norm{P_f}$ and $\norm{P_b}\leq1$, it follows that $\bm{u}^T(2I - \frac{1}{2}\left([P_f+P_f^T] + [P_b+P_b^T]\right)\bm{u} = \bm{u}^T \left(A_\parallel + A_\parallel^T\right) \bm{u}\geq0$.
Therefore choosing $\alpha<0$ gives $\bm{u}^T\left((HP_\parallel)+(HP_\parallel)^T\right)\bm{u} = \frac{\kappa_\parallel\alpha}{2\Delta{x}}\bm{u}^T \left(A_\parallel + A_\parallel^T\right) \bm{u}\le 0$. 
\end{proof}
We now prove the stability of the semi-discrete approximation \eqref{eq:semi-discrete approximation}.
\begin{theorem}\label{eq:semi-discrete stability}
Consider the semi-discrete approximation \eqref{eq:semi-discrete approximation} for homogeneous boundary data $\bm{g} =0$ where the numerical parallel diffusion operator $P_\parallel$ and the $SAT$ are given by \eqref{eq:discrete parallel operator and SAT} , with $\tau_\parallel = {\alpha}/{\Delta x} \le 0$ and $\tau_0 = -1$. Let $\norm{\bm{u}}_H^2 = \bm{u}^TH\bm{u}$,
 if $\|P_f\|\le 1$ and $\|P_b\|\le 1$ then
$$
\dv{t}\norm{\bm{u}}_H^2 \le 0, \quad \forall \bm{u}\in \mathbb{R}^n.
$$
\end{theorem}
\begin{proof}
 Multiply \eqref{eq:semi-discrete approximation} from the left by $\bm{u}^TH$, we have
\begin{align}
    \bm{u}^TH\dv{\bm{u}}{t} &= -\bm{u}^T(M^{(\kappa)} +  BKD_x)\bm{u} + \tau_0\bm{u}^T BK D_x \bm{u}  + \bm{u}^T H P_\parallel \bm{u}.
\end{align}
Choosing $\tau_0=-1$ and adding the transpose  of the products gives,
\begin{align}
     \dv{t}\norm{\bm{u}}_H^2 &= 
        -2\bm{u}^TM^{(\kappa)}\bm{u}
        + \bm{u}^T \left( (H P_\parallel) + (H P_\parallel)^T\right) \bm{u} \le 0.
\end{align}
\end{proof}
\subsection{The fully-discrete approximation}\label{sec:sub:The fully-discrete approximation}

We discretise the time variable $t_{l+1} = t_l + \Delta{t}_l$ with the time-step $\Delta{t}_l >0$ where $t_0 =0$ and $l = 0, 1, 2, \cdots$. The fully discrete solution at the time level $t_l>0$ is denoted $\bm{u}^l$ with $\bm{u}^0 = \bm{f}$.
Solving the fully discrete version of the semi-discrete anisotropic diffusion equation \eqref{eq:semi-discrete approximation} is performed by operator splitting. This results in a two stage solve,
\begin{align}
\label{eq:operator split stage 1}
    &\left(I+\Delta t H^{-1} M_{x}^{(k)}\right) \bm{u}^{l+\half} = \bm{u}^l + \Delta{t}\bm{F}(t_{l+1}), \quad  \bm{F}(t_{t_{l+1}}) = \tau_0 H^{-1} B\bm{g}(t_{l+1})  \\
    & \bm{w}_f^{l+\half} = P_f\bm{u}^{l+\half}, \quad \bm{w}_b^{l+\half} = P_b\bm{u}^{l+\half},
\label{eq:operator split stage 2 w}
    \\
\label{eq:operator split stage 2}
    &\bm{u}^{l+1}         = \bm{u}^{l+\half} + \frac{\Delta{t}\tau_\parallel \kappa_\parallel}{2}H^{-1}  {\left(\bm{u}^{l+1} - \frac{1}{2}[\bm{w}_f^{l+\half} + \bm{w}_b^{l+\half}]\right)}. 
\end{align}
The first stage is the perpendicular solve is a backward Euler approximation and involves solving a elliptic linear system which can be solved efficiently by the conjugate gradient method. Stage two, which includes \eqref{eq:operator split stage 2 w} and \eqref{eq:operator split stage 2}, propagates the parallel diffusion and can be computed directly.
\section{Numerical results}\label{sec:Results}
We first demonstrate the convergence of the SBP-SAT scheme (without parallel component) by the method of manufactured solutions \cite{steinberg_symbolic_1985,roache_code_2001}. We choose the manufactured solution, with the exact solution
\begin{align}
    {u}(x,t) = \cos(2\pi t)\sin(17\pi x + 1),
\end{align}
The convergence results are shown in Figure \ref{fig:MMS convergence}. We set a fixed time step of $\Delta t = \Delta x^2/100$. Comparison with the (dashed) reference lines shows both slightly over-perform their expected convergence rate with $\sim\mathcal{O}(h^{2.5})$ for the second order, and $\sim\mathcal{O}(h^{4.5})$ for the fourth order.
\begin{figure}[t!]
    \centering
    \begin{tikzpicture}
        \begin{axis}[ymode=log,xmode=log,
            xlabel={grid size}, ylabel={relative error},
            legend pos=south west,
            width=0.67\textwidth]
            \addplot table[x index=0,y index=2,col sep=tab] {figures/1DMMSTests_O2.csv};
            \addplot table[x index=0,y index=2,col sep=tab] {figures/1DMMSTests_O4.csv};
            
            \logLogSlopeTriangle{0.95}{0.4}{0.65}{2.0}{blue,dashed};
            \logLogSlopeTriangle{0.95}{0.4}{0.5}{4.0}{red,dashed};
            
            \addlegendentry{2nd order};
            \addlegendentry{4th order};
        \end{axis}
    \end{tikzpicture}
    \caption{Convergence rates for second (blue) and fourth (red) order summation by parts operators with first order time solver. Dashed lines are references lines and have the expected slopes for associated convergence rates.
    }
    \label{fig:MMS convergence}
\end{figure}
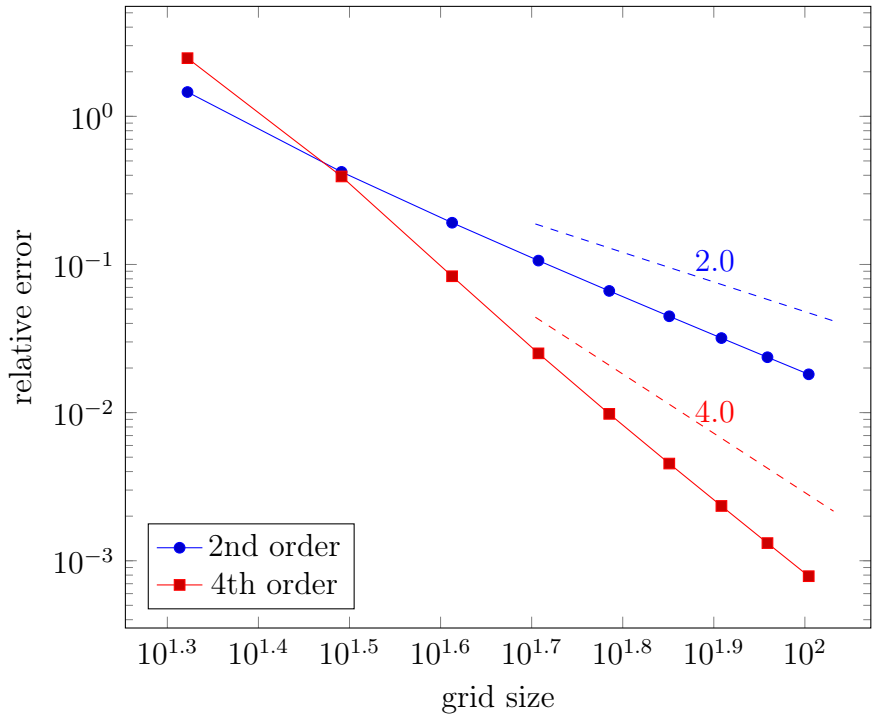

\subsection{Examples}
We now present examples of the field aligned 1D anisotropic diffusion equation, which demonstrate the effect of the the parallel operator on the solution. These are shown in Figures \ref{fig:example 1} and \ref{fig:example random map}. In all cases the boundary conditions are no flux, $\partial_x u|_{x=0} = \partial_x u|_{x=1} = 0$, and the diffusion coefficients in the perpendicular and parallel directions are $\kappa_\perp=10^{-3}$ and $\kappa_\parallel=1$.

The initial condition for examples in Figure \ref{fig:example 1} is a Gaussian,
\begin{align}\label{eq:example exp IC}
    u(x,0) = \exp\left( \frac{-(x-0.5)^2}{0.02} \right).
\end{align}
The parallel map in the forward and backward directions on the left and right are,
\begin{align}
\label{eq:example mapping}
    F_1(x) = 1 - \exp(-x)    \quad\text{and}\quad
    F_2(x) = \frac{1}{2}(\tanh(2\pi x - \pi) + 1),
\end{align}
respectively. The point mapping is visualised in the top row of Figure \ref{fig:example 1}.
\begin{figure}[h]
    \centering
\frame{
    \includegraphics[width=0.95\textwidth]{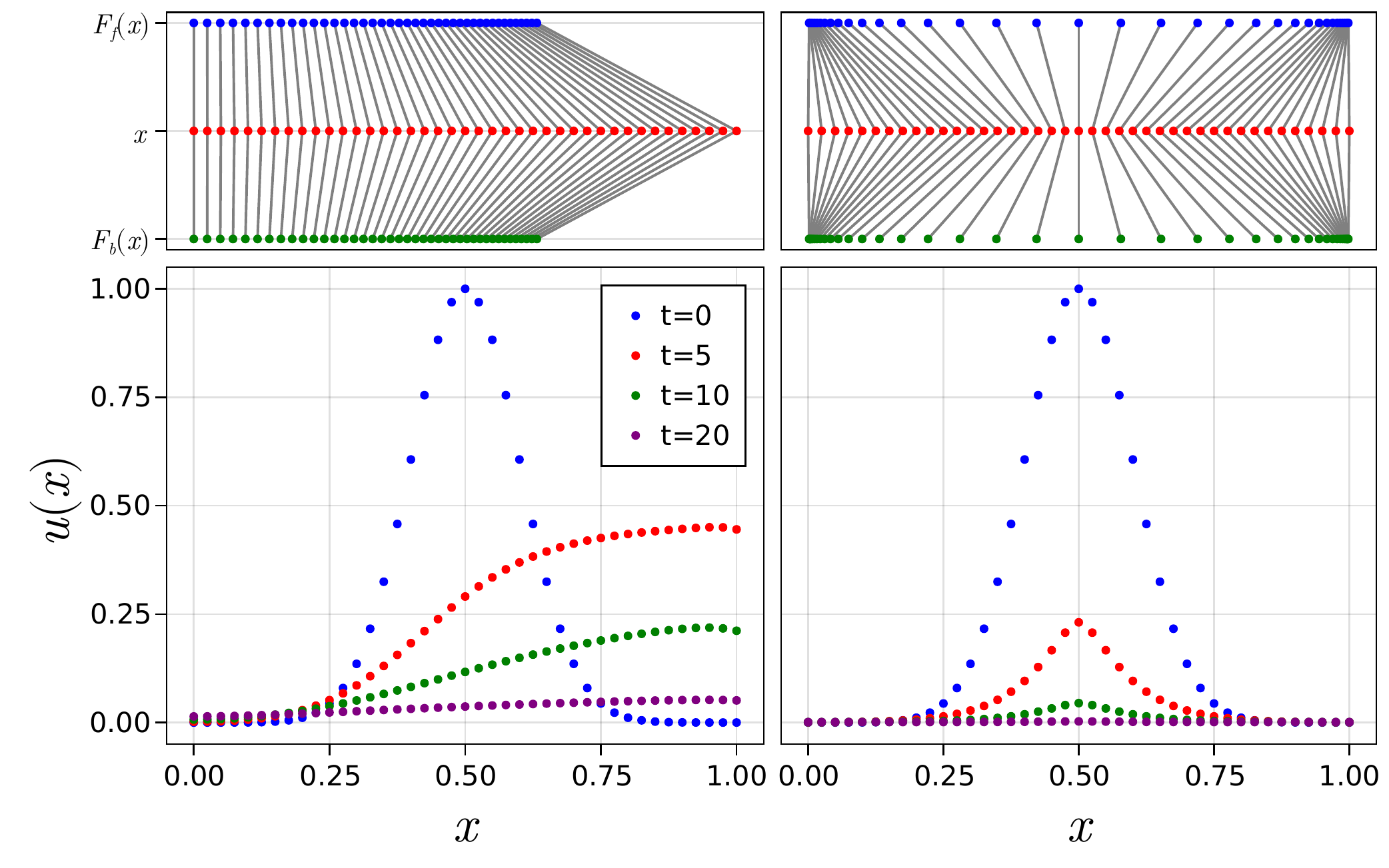}
}
    \caption{
    Top left: Parallel mapping given by $F_1$ in equation \eqref{eq:example mapping}.
    Bottom left: Solution to \eqref{eq:ADE 1.5D} at various times with parallel map given by equation $F_1$ in \eqref{eq:example mapping}.
    Top right: point mapping as per $F_2$ in \eqref{eq:example mapping}.
    Bottom right: Solution to \eqref{eq:ADE 1.5D} at various times with parallel map given by equation $F_2$ in \eqref{eq:example mapping}.
    }
    \label{fig:example 1}
\end{figure}
Solutions in Figure \ref{fig:example 1} tend towards a uniform one as expected with no-flux boundaries. Given the point mapping by $F_1$, we see the right hand side of the solution maps to the centre of the Gaussian profile, which explains the increase in $u$ on the right hand side. The point mapping by $F_2$ diffuses into low $u$ regions, flattening out the solution quickly.

Examples in Figure \ref{fig:example random map} demonstrate both the robustness of the approach and the effect of the operator on a standard 1D solution to the equation. Here the forward and backward maps randomly map points in the domain. Equations $F_1$ and $F_2$ in \eqref{eq:example mapping} are now used as the initial conditions in the left and right figures respectively.
\begin{figure}[h]
    \centering
\frame{
    \includegraphics[width=0.95\columnwidth]{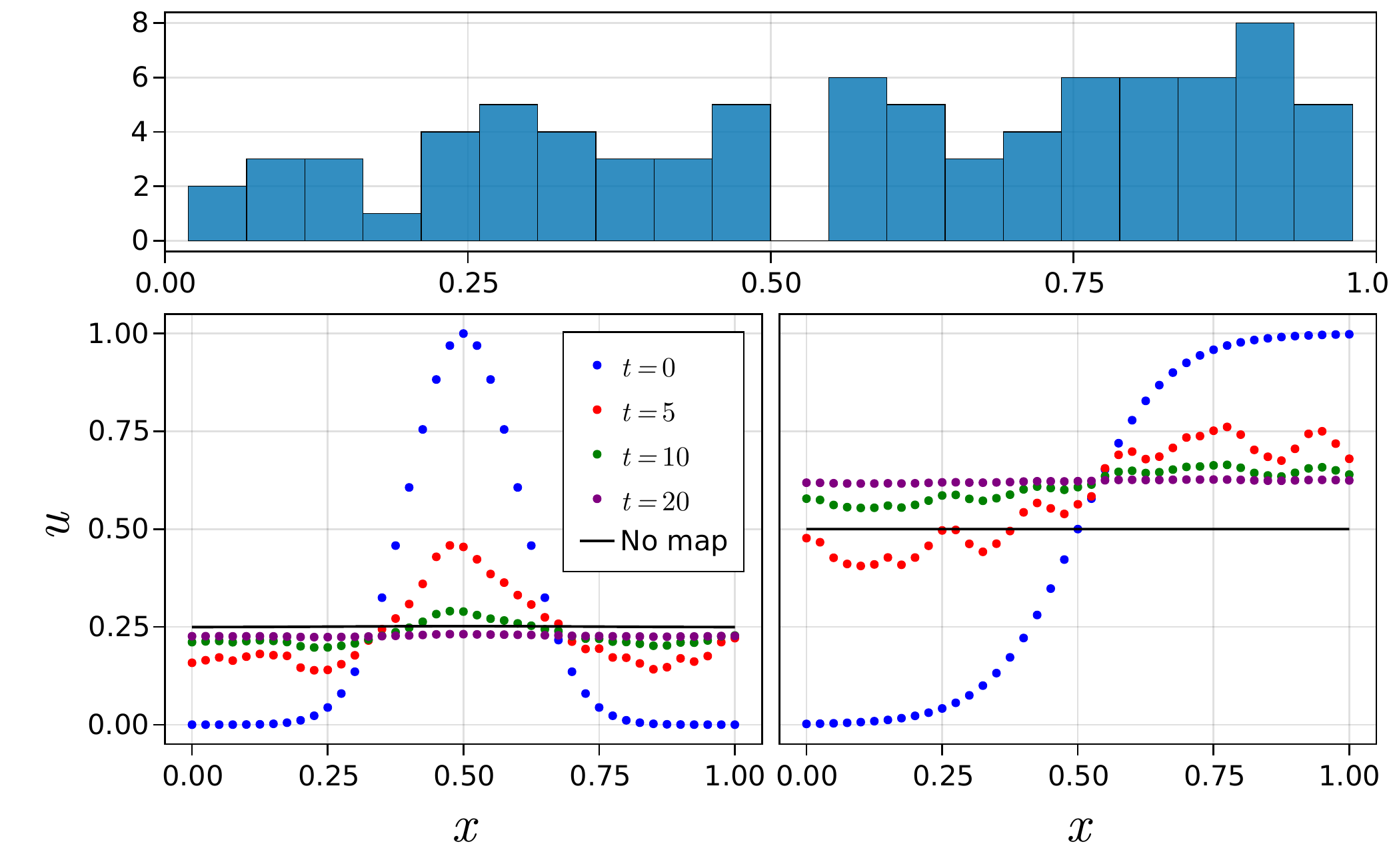}
}
    \caption{\textit{Top}: Distribution of points on forward and backward planes, showing a slight bias towards the right side of the domain. 
    \textit{Bottom left}: Evolution of solution with random point mapping and initial condition specified by $F_1$ in \eqref{eq:example mapping}.
    \textit{Bottom right}: Same as right figure, but initial condition given by $F_2$ in \eqref{eq:example mapping}.
    The black line in the lower figures corresponds to the 1D solution without parallel mapping.}
    \label{fig:example random map}
\end{figure}

Solutions are again uniform as expected and can be compared to a solution with no parallel mapping (black line). They deviate from this solution because the parallel parallel map has a slight bias towards right side of the domain as shown in the top of Figure \ref{fig:example random map}. In the case of the Gaussian, this results in diffusion into the low $u$ region, reducing the final state of the solution. In the sigmoid function, points are mapped into the high $u$ region, resulting in a slightly higher final solution.

\section{Conclusions}\label{sec:Conclusions}

We have derived a stable and efficient numerical method to solve the anisotropic diffusion equation in 2D geometry not aligned with a regular mesh by using an operator splitting technique. We achieve this with the use of SBP with SAT in the perpendicular solve and an integral operator for the parallel solve. The time steps are evolved by using an implict Euler conjugate gradient method.

Our method produces accurate results, verified by use of the method of manufactured solutions. The results show the SBP-SAT method over-performs the expected second and fourth order convergence rates.
Examples show our method solving the field aligned anisotropic diffusion equation with a variety of parallel maps. Moreover, the random point mapping examples shows our approach is robust, since the method returns expected results.

Future work will extend this method to 2D planes in the perpendicular direction and use parallel maps given by systems of ordinary differential equations. Specifically we will be interested in extending the method to geometry provided by real magnetic fields.

\printbibliography

\end{document}

%% file: definitions.tex
\usepackage[utf8]{inputenc}
\usepackage[T1]{fontenc}
\usepackage{graphicx}
\usepackage{epstopdf}
\usepackage[linesnumbered,boxed]{algorithm2e}
\usepackage[margin=20mm]{geometry} 
\usepackage{amssymb}
\usepackage{latexsym}
\usepackage{listings}
\usepackage{amsfonts}
\usepackage{amsthm}
\usepackage{bm}
\usepackage{amsmath}
\usepackage{mathtools}
\usepackage{thmtools}
\usepackage[linesnumbered]{algorithm2e}
\usepackage{biblatex}
\graphicspath{ {Figures/} } 
\usepackage{color}
\usepackage{standalone} 
\usepackage[table,x11names]{xcolor}
\usepackage{soul}
\usepackage{multirow}
\usepackage{multicol}			
\usepackage{caption}			
\usepackage{subcaption}		
\usepackage[colorlinks]{hyperref}
\hypersetup{citecolor = {blue},pdfauthor=author}
\usepackage{tikz}
\usepackage{pgfplots}
\usepackage{sectsty}			
\usepackage{physics}
\usepackage{array}
\usepackage{float}
\usepackage{forest}
\usepackage{xparse}
\usepackage{lipsum}
\usepackage{multicol}
\usepackage{todonotes}

\usetikzlibrary{shapes,arrows,arrows.meta,decorations.pathmorphing,decorations.markings,decorations.pathreplacing,patterns,shapes.geometric}
\pgfplotsset{compat=1.12}
\newcommand{\R}{\mathbb{R}}

\newcommand{\half}{\frac{1}{2}}

\newcolumntype{C}{>{$}c<{$}}

\newtheorem{theorem}{Theorem}[section]
\theoremstyle{definition}
\newtheorem{lemma}{Lemma}

\declaretheoremstyle[
  numbered=yes,
  numberlike=theorem,
  spaceabove=1em plus 0.75em minus 0.25em,
  spacebelow=1em plus 0.75em minus 0.25em,
  qed={}
]{exmpstyle}

\newcommand{\logLogSlopeTriangle}[5]
{

    \pgfplotsextra
    {
        \pgfkeysgetvalue{/pgfplots/xmin}{\xmin}
        \pgfkeysgetvalue{/pgfplots/xmax}{\xmax}
        \pgfkeysgetvalue{/pgfplots/ymin}{\ymin}
        \pgfkeysgetvalue{/pgfplots/ymax}{\ymax}

        \pgfmathsetmacro{\xArel}{#1}
        \pgfmathsetmacro{\yArel}{#3}
        \pgfmathsetmacro{\xBrel}{#1-#2}
        \pgfmathsetmacro{\yBrel}{\yArel}
        \pgfmathsetmacro{\xCrel}{\xArel}

        \pgfmathsetmacro{\lnxB}{\xmin*(1-(#1-#2))+\xmax*(#1-#2)} 
        \pgfmathsetmacro{\lnxA}{\xmin*(1-#1)+\xmax*#1} 
        \pgfmathsetmacro{\lnyA}{\ymin*(1-#3)+\ymax*#3} 
        \pgfmathsetmacro{\lnyC}{\lnyA+(-#4)*(\lnxA-\lnxB)}
        \pgfmathsetmacro{\yCrel}{\lnyC-\ymin)/(\ymax-\ymin)} 

        \coordinate (A) at (rel axis cs:\xArel,\yArel);
        \coordinate (B) at (rel axis cs:\xBrel,\yBrel);
        \coordinate (C) at (rel axis cs:\xCrel,\yCrel);

        \draw[#5]   (A) node[pos=0.5,anchor=north] {1}
                    (B)-- 
                    (C) node[xshift=0.2,yshift=0.2,pos=0.6,above] {#4};
    }
}